\documentclass[10pt,bezier]{article}
\usepackage{amsmath,amssymb,amsfonts}

\newtheorem{thm}{Theorem}[section]

\newtheorem{cor}[thm]{Corollary}
\newtheorem{lem}[thm]{Lemma}
\newtheorem{exm}[thm]{Example}

\newenvironment{proof}[1]{{\bf proof. }{\rm #1}}{\hfill $\rule {2mm}{2mm}$\\}
\begin{document}

\title{The 3-way intersection problem for $S(2,4,v)$ designs }
\author{
{\sc Saeedeh Rashidi  AND Nasrin Soltankhah}
 \\[5mm]
Department of Mathematics\\ Alzahra University  \\
Vanak Square 19834 \ Tehran, I.R. Iran \\
{s.rashidi@alzahra.ac.ir}\\
{ soltan@alzahra.ac.ir}\\
\date{}
}

\maketitle


\begin{abstract}
In this paper the 3-way intersection problem for $S(2,4,v)$ designs is investigated. Let $b_{v}=\frac {v(v-1)}{12}$ and $I_{3}[v]=\{0,1,\dots,b_{v}\}\setminus\{b_{v}-7,b_{v}-6,b_{v}-5,b_{v}-4,b_{v}-3,b_{v}-2,b_{v}-1\}$.
Let $J_{3}[v]=\{k|$ there exist three $S(2,4,v)$ designs with $k$ same common blocks$\}$.
We show that $J_{3}[v]\subseteq I_{3}[v]$ for any positive integer $v\equiv1, 4\ (\rm mod \ 12)$ and $J_{3}[v]=I_{3}[v]$, for $ v\geq49$ and $v=13 $.
We find $J_{3}[16]$ completely.
Also we determine some values of $J_{3}[v]$ for $\ v=25,28,37$ and 40.
\end{abstract}
\hspace*{-2.7mm} {\bf KEYWORDS:} {  \sf 3-way intersection; $S(2,4,v)$ design; GDD; trade}
%
\section{Introduction}
A $ Stiener\  system \ S(2,4,v)$ is a pair $(\mathcal{V},\mathcal{B})$ where $\mathcal{V}$  is a $v$-element set and $\mathcal{ B}$
 is a family of $4$-element subsets of $\mathcal{V}$ called $blocks$, such that each $2$-element subsets of $\mathcal{V}$ is contained in
 exactly one block.

Two Steiner systems $S(2,4,v),\ (\mathcal{V},\mathcal{B})$ and $\  (\mathcal{V},\mathcal{B}_{1})$ are said to $intersect$ in $s$ blocks if
$|\mathcal{B}\cap\mathcal{B}_{1}|=s$. The intersection problem for $S(2,4,v)$ designs can be extended in this way: determine the sets
$\overline{J_{\mu}[v]}(J_{\mu}[v])$ of all integers $s$ such that there exists a collection of $\mu \ (\geq 2)\ S(2,4,v)$ designs mutually
intersecting in $s$ blocks (in the same set of $s$ blocks). This generalization is called $\mu$-way intersection problem.
Clearly $\overline{J_{2}[v]}=J_2[v]=J[v]$ and $J_{\mu}[v]\subseteq\overline{J_{\mu}[v]}\subseteq J[v]$.

The intersection problem for $\mu=2$ was considered by Colbourn, Hoffman, and  Lindner in~\cite{2}. They determined
the set $J[v](J_{2}[v])$
completely for all values $v\equiv1, 4\ (\rm mod \ 12)$, with some possible exceptions for $v=$ 25, 28 and 37. Let $[a,b]=\{a,a+1,...,b-1,b\}$, $b_{v}=\frac {v(v-1)}{12}$,
and $I[v]=[0,b_{v}]\setminus( [b_{v}-5,b_{v}-1]\cup \{b_{v}-7\})$. It is shown in~\cite{2}; that:

$(1)\  J[v]\subseteq I[v]$ for all  $v\equiv1, 4\ (\rm mod\  12)$.

$(2)\ J[v]=I[v]$ for all admissible $v\geq40$.

$(3)\ J[13]=I[13]$ and $J[16]=I[16]\setminus\{7,9,10,11,14\}$.

$(4)\ I[25] \setminus \{31,33,34,37,39,40,41,42,44\} \subseteq J[25]$ and $\{42,44\} \not \subseteq J[25]$.

$(5)\ I[28]\setminus\{44,46,49,50,52,53,54,57\}\subseteq J[28]$.

$(6)\ I[37]\setminus(\{64,66,76,82,84,85,88\} \cup [90,94] \cup [96,101])\subseteq J[37]$.

Also Chang, Feng, and Lo Faro investigate another type of
intersection which is called triangle intersection (See~\cite{14}).
Milici and Quattrocchi~\cite{4} determined $J_{3}[v]$ for $STS$s.
Other results about the intersection problem can be found in
$\cite{22,23,24,25,26,27,28}$. In this paper we investigate the
three way intersection problem for $S(2,4,v)$ designs. We set
$I_{3}[v]=[0,b_{v}]\setminus[b_{v}-7, b_{v}-1]$.  As our main
result, we prove the following theorem.
\begin{thm}~\label{1.1}
$(1)\  J_{3}[v]\subseteq I_{3}[v]$ for all  $v\equiv1, 4\ (\rm mod\  12)$.

$(2)\ J_{3}[v]=I_{3}[v]$ for all admissible $v\geq49$.

$(3)\  I_{3}[40]\setminus�\{ b_{40}-15, b_{40}-14\}\subseteq J_{3}[40]$.

$(4)\ J_{3}[13]=I_{3}[13]$ and $J_{3}[16]=I_{3}[16]\setminus\{7,9,10,11,12\}$..

$(5)\ [0,11] \cup \{13,15,17,20,29,50\} \cup [22,24] \subseteq J_{3}[25]$ and $\{42\}\not \in J_{3}[25]$.

$(6)\ [1,24]\cup \{27,28,33,37,39,63\}\subseteq J_3[28]$.

$(7)\ \{18,19,78,79,81,87,102,103,111\} \cup [21,32] \cup [34,36] \cup [38,43] \cup [45,48] \cup [52,54] \cup [58,63] \cup [67,71] \subseteq J_{3}[37]$.
\end{thm}
%
\section{Necessary conditions}
In this section we establish necessary conditions for $J_3[v]$. For this purpose, we use another concept that is relative to intersection problem:
A $(v,k,t)$ $trade$ of volume $s$ consists of two disjoint collections  $ T_{1}$ and ${T}_2$, each of $s$ blocks,
such that for every $t$-subset of blocks, the  number of blocks containing these elements ($t$-subset) are the same in both  $T_{1}$ and ${T}_2$.
A  $(v,k,t)$ trade of volume $s$ is $Steiner$ when for  every $t$-subset of blocks, the
number of blocks containing these elements are at most one.
A $\mu$-way $(v,k,t)$ trade ${T}=\{{T}_{1},{T}_{2},\ldots,{T}_{\mu}\}$, $\mu\geq2$ is a set of pairwise disjoint $(v,k,t)$ trade.\\
In every collection the union of blocks must cover the same set of elements. This set of elements is called the \textit{foundation} of the trade. Its notation is found (T) and
$r_{x}=$ no. of blocks in a collection which contain the element $x$.\\
By definition of the trade, if $b_v-s$ is in $J_{3}[v]$, then it is clear that there exists a 3-way  Steiner $(v,4,2)$ trade of volume $s$. Consider three
$S(2,4,v)$ designs (systems) intersecting in $b_v-s$ same blocks (of size four). The remaining set of blocks (of size four) form disjoint partial quadruple systems,
containing precisely the same pairs, and each has $s$ blocks. Rashidi and  Soltankhah in~\cite{15} established that there do not exist a 3-way Steiner $(v,4,2)$ trade
of volume $s$, for  $s\in \{1,2,3,4,5,6,7\}$.
So we have the following lemma:
\begin{lem}~\label{2.1}
$J_{3}[v]\subseteq I_3[v]$.
\end{lem}
%
\section{Recursive constructions}
In this section we give some recursive constructions for the 3-way intersection problem.
The concept of GDDs plays an important role in these constructions.
Our aim of common blocks is the same common blocks in the sequel.

Let $K$ be a set of positive integers . A $group\ divisible\ design$ $K$-GDD (as GDD for short) is a triple $(\mathcal{X},\mathcal{G},\mathcal{A})$  satisfying the following
properties: (1) $\mathcal{G}$ is a partition of a finite set $\mathcal{X}$ into subsets (called groups); (2) $\mathcal{A}$ is a set of subsets of $\mathcal{X}$
(called blocks), each of cardinality from $K$, such that a group and  a block contain at most one common element; (3) every pair of elements from distinct groups
occurs in exactly one  block.

If $\mathcal{G}$ contains $u_i$ groups of size $g_i$, for $1\leq i\leq s$, then we denote by $g_{1}^{u_{1}}g_{2}^{u_{2}}\ldots g_{s}^{u_{s}}$ the \textit{group type} (or type) of the GDD.
If $K=\{k\}$, we write $\{k\}$-GDD as $k$-GDD. A $K$-GDD of type $1^v$ is commonly called a $pairwise\ balanced\ design$, denoted by $(v,K,1)$-PBD.
When $K=\{k\}$ a PBD is just a Steiner system $S(2,k,v)$.\\
The following construction is a variation of Willson's Fundamental Construction.
\begin{thm}~\label{3.1}
($Weighting\ construction$). Let  $(\mathcal{X},\mathcal{G},\mathcal{A})$ be a GDD with groups $G_{1},\ G_{2},\ldots,\ G_{s} $. Suppose that there exists a function $w\ :\ X \rightarrow \ Z^{+}\cup\{0\}$ (a weight function) so that for each block $A=\{x_{1},\ldots, x_{k}\}\in \mathcal{A}$ there exist three $K$-GDDs of type $[w(x_{1}),\ \ldots\ w(x_{k})]$ with $b_{A}$ common blocks. Then there exist three $K$-GDDs of type $\textbf{[}\sum_{x\in G_{1}}{w(x)},\ldots,\sum_{x\in G_{s}}{w(x)}\textbf{]}$ which intersect in $\sum_{A\in\mathcal{A} }{b_{A}}$ blocks.
\end{thm}
\begin{proof}
For every $x\in \mathcal{X}$, let $S(x)$ be a set of $w(x)$ ``copies'' of $x$. For any $\mathcal{ Y}\subset \mathcal{X} $, let
$S(\mathcal{ Y})=\bigcup_{y\in\mathcal{ Y}}{S(y)}$. For every block $A\in \mathcal{A}$, there exist three $K$-GDDs: $(S(A),\ \{S(x)\ :\  x\in A\},\ \textit{B}_{A})$,  $(S(A),\ \{S(x)\ :\  x\in A\},\ \dot{\textit{B}}_{A})$, $(S(A),\ \{S(x)\ :\  x\in A\},\ \ddot{\textit{B}}_{A})$, which intersect in $b_{A}$
blocks. Then it is readily checked that there exist three, $K$-GDDs: $(S( \mathcal{X}),\ \{S(G)\ :\  G\in\mathcal{G} \},\ \cup_{A\in \mathcal{A}} \textit{B}_{A})$,
$(S( \mathcal{X}),\ \{S(G)\ :\  G\in\mathcal{G}\},\ \cup_{A\in \mathcal{A}}\dot{\textit{B}}_{A})$, $(S( \mathcal{X}),\ \{S(G)\ :\  G\in \mathcal{G}\},\ \cup_{A\in \mathcal{A}}\ddot{\textit{B}}_{A})$, which intersect in $\sum_{A\in \mathcal{A}}b_{A}$ blocks.
\end{proof}
\begin{thm}~\label{3.2}
($Filling\ construction\ (i)$). Suppose that there exist three 4-GDDs of type $g_{1}g_{2}\ldots g_{s}$ which intersect in $b$ blocks. If there exist three $S(2,4,g_{i}+1)$ designs with $b_{i}$ common blocks for $1\leq i\leq s$, then there exist three $S(2,4,\sum_{i=1 }^{s}{g_{i}} +1)$ designs with $b+\sum_{i=1 }^{s} {b_{i}}$ common blocks.
\end{thm}
\begin{proof}
It is obvious.
\end{proof}
\begin{thm}~\label{3.3}
($Filling\ construction\ (ii)$). Suppose that there exist three 4-GDDs of type $g_{1}g_{2}\ldots g_{s}$ which intersect in $b$ blocks. If there exist three $S(2,4,g_{i}+4)$ designs containing $b_{i}$ common blocks for $1\leq i\leq s$. Also all designs containing a block $y$. Then there exist three $S(2,4,\sum_{i=1 }^{s}{g_{i}} +4)$ designs with $b+\sum_{i=1 }^{s} {b_{i}}- (s-1)$ common blocks.
\end{thm}
\begin{proof}
Let $(\mathcal{X},\mathcal{G},\mathcal{A}_1)$ , $(\mathcal{X},\mathcal{G},\mathcal{A}_2)$ and $(\mathcal{X},\mathcal{G},\mathcal{A}_3)$ be three 4-GDDs of type $g_{1}g_{2}\ldots g_{s}$ which intersect in $b$ blocks. Let $\mathcal{Y}=\{y_1,y_2,y_3,y_4\}$ be a set of cardinality 4 such that $\mathcal{X}\cap\mathcal{Y}=\phi$.\\
For $1\leq i\leq s$, there exist three $S(2,4,g_i+4)$ designs $(g_i\cup\mathcal{Y}, \varepsilon_{1i})$, $(g_i\cup\mathcal{Y}, \varepsilon_{2i})$ and $(g_i\cup\mathcal{Y}, \varepsilon_{3i})$ containing the same block $y=y_1,y_2,y_3,y_4$ with $b_i$ common blocks. It is easy to see that
$(\mathcal{X}\cup \mathcal{Y}, \mathcal{A}_1\cup {(\bigcup _{1\leq i\leq s-1}{(\varepsilon_{1i}-y)}})\cup {\varepsilon_{1s})}$, $(\mathcal{X}\cup \mathcal{Y}, \mathcal{A}_2\cup {(\bigcup _{1\leq i\leq s-1}{(\varepsilon_{2i}-y)}})\cup {\varepsilon_{2s})}$ and  $(\mathcal{X}\cup\mathcal{Y}, \mathcal{A}_3\cup {(\bigcup _{1\leq i\leq s-1}{(\varepsilon_{3i}-y)}})\cup {\varepsilon_{3s})}$ are three $S(2,4,\sum_{i=1 }^{s}{g_{i}} +4)$ designs with $b+\sum_{i=1 }^{s} {b_{i}}- (s-1)$ common blocks.
 �\end{proof}
We apply another type of recursive constructions that explained in  the following.\\
Let there be three $S(2,4,v)$ designs with a common parallel class, then  $J_{p3}[v]$ for $v\equiv4\ (\rm{mod}\ 12)$ denotes  the number of blocks shared by these $S(2,4,v)$ designs , in addition to those shared in the parallel class.\\
\begin{lem}~\label{3.4}
Let $G$ be a GDD on $v=3s+6t$ elements with $b$ blocks of size 4 and group type $3^{s}6^{t}$, $s\geq1$.  For $1\leq i\leq b$, let $a_{i}\in J_{p3}[16]$. For $1\leq i\leq s-1$, let ${c_{i}+1}\in J_{3}[16]$ and let $c_{s}\in J_{3}[16]$. For $1\leq i\leq t$, let ${d_{i}+1}\in J_{3}[28]$. Then there exist three $S(2,4,4v+4)$ designs with precisely
${\sum_{i=1}^{b}{a_{i}}}+{\sum_{i=1}^{s}{c_{i}}}+{\sum_{i=1}^{t}{d_{i}}}$ blocks in common.
\end{lem}
\begin{proof}
The proof is similar to Lemma 3.3 in~\cite{2}.
\end{proof}
The $flower$ of an element is the set of blocks containing that element. Let $J_{f3}[v]$ denote the number of blocks shared by three $S(2,4,v)$ designs, in addition to those in a required common flower.
\begin{lem}~\label{3.5}
Let $G,\ B$ be a GDD of order $v$ with $b_4$ blocks of size 4,  $b_5$ blocks of size 5 and group type $4^{s}5^{t}$.  For $1\leq i\leq b_4$, let $a_{i}\in J_{f3}[13]$.
For $1\leq i\leq b_5$, let $c_{i}\in J_{f3}[16]$. For $1\leq i\leq s$, let ${d_{i}}\in J_{3}[13]$ and for $1\leq i\leq t$, let ${e_{i}}\in J_{3}[16]$. Then there exist three $S(2,4,3v+1)$ designs intersecting in  precisely
${\sum_{i=1}^{b_4}{a_{i}}} + {\sum_{i=1}^{b_5}{c_{i}}}+{\sum_{i=1}^{s}{d_{i}}}+{\sum_{i=1}^{t}{e_{i}}}$ blocks.
\end{lem}
\begin{proof}
The proof is similar to Lemma 3.5 in~\cite{2}.
\end{proof}
\begin{lem}~\label{3.6}~\cite{1}. The necessary and sufficient conditions for the existence of a 4-GDD of type $g^n$ are: (1) $n\geq4$, (2) $(n-1)g\ \equiv\ 0 \ (\rm{mod}\ 3)$, (3) $n(n-1)g^{2} \ \equiv\ 0 \ (\rm{mod}\ 12)$, with the exception of $(g,n)\in \{(2,4), (6,4)\}$, in which case no such GDD exists.
\end{lem}
\begin{lem}~\label{3.7}~\cite{5}. There exists a $(v, \{4,7^{*}\},1)$-PBD with exactly one block of size 7 for any positive integer $v\equiv 7,10 \ (\rm{mod} \ 12)$ and $v\neq10,19$.
\end{lem}
\begin{lem}~\label{3.8}~\cite{1}. A 4-GDD of type $12^um^1$ exists if and only if either $u=3$ and $m=12$, or $u\geq4$ and $m\equiv 0\ (\rm mod\ 3)$ with $0\leq m\leq 6(u-1)$.
\end{lem}
%
 \section{Ingredients }
In this section we discuss some small cases needed for general constructions.
�\begin{lem}~\label{4.1}
$J_{3}[13]=I_{3}[13]$.
\end{lem}
\begin{proof}
Construct an  $S(2,4,13)$ design,  $(\mathcal{V},\mathcal{B})$ with $\mathcal{V}=\mathcal{Z}_{10}\cup\{a,b,c\}$. All blocks of $\mathcal{B}$ are listed in the following,
which can be found in Example~1.26 in~\cite{1}.\\
  0 0 0 0 1 1 1 2 2 3 3 4 5\\
 1 2 4 6 2 5 7 3 6 4 7 8 9\\
 3 8 5 a 4 6 b 5 7 6 8 9 a\\
 9 c 7 b a 8 c b 9 c a b c\\
Consider the following permutations on $\mathcal{V}$.
\begin{center}
\begin{tabular}{|c|c|c|c|}
\hline
$\pi_{1}$& $\pi_{2}$ & $\pi_{3}$ &  int. no. \\
\hline\hline
id & $(0,1,2,3,4,5)$  & $(5,4,3,2,1,0)$ &0\\
id & $(8,5)(a,b)(3,7)(1,6)$  & $(8,6)(1,5)(a,3,b,7)$ &1\\
id  & $(7,b,c,6)$  & $(8,5,6,7)$ &2\\
id& $(4,7)(9,2)(1,8) $ & $(3,9,c)(1,8)$ &3\\
id & $(3,7)(c,0,2)(1,6)(9,b)$  & $(9,4)(3,7)(0,2)(1,6)$ &4\\
id & $(a,b)(4,5)$  & $(a,b)(c,8)$ &5\\
id & id & id &13\\
\hline
\end{tabular}
\end{center}
\end{proof}
\begin{lem}~\label{4.2}
$J_{3}[16]=I_{3}[16]\setminus\{7,9,10,11,12\}$.
\end{lem}
\begin{proof}
The proof has three steps:

(1) $J_{3}[16]\subseteq J[16]=\{0,1,2,3,4,5,6,8,12,20\}$.

(2) Construct an  $S(2,4,16)$ design,  $(\mathcal{V},\mathcal{B})$ with $\mathcal{V}=\mathcal{Z}_{10}\cup\{a,b,c,d,e,f\}$. All 20 blocks of $\mathcal{B}$ are listed in the following,
which can be found in Example~1.31 in~\cite{1}.\\

$\{0,1,2,3\},\  \{0,4,5,6\},\ \{0,7,8,9\},\ \{0,a,b,c\},\ \{0,d,e,f\},\ \{1,4,7,a\},\\
\{1,5,b,d\},\ \{1,6,8,e\},\ \{1,9,c,f\},\ \{2,4,c,e\},\ \{2,5,7,f\},\ \{2,6,9,b\},\\
\{2,8,a,d\},\ \{3,4,9,d\},\ \{3,5,8,c\},\ \{3,6,a,f\},\ \{3,7,b,e\},\ \{4,8,b,f\},\\
\{5,9,a,e\},\ \{6,7,c,d\}$.\\

Consider the following permutations on $\mathcal{V}$.\\
\begin{center}
\begin{tabular}{|c|c|c|c|}
\hline
$\pi_{1}$& $\pi_{2}$ & $\pi_{3}$ & int. no.\\
\hline\hline
id & $(0,1,4,9,d)(6,5)(b,f)$  & $(2,3,7,6,c)(5,8)(a,b)(f,4)$ &0\\
id  & $(0,1,2,3)(8,9,5,b,c)$  & $(d,f,e,a)(4,6,7)(c,9)$ &1\\
id  & $(1,2)(a,b)(7,f)(c,e)(6,8)$  & $(a,b,7,f,c,e,6,8)(4,d)$ &2\\
id & $(c,e)(5,6)(b,f,7,d,9) $ & $(8,b,7,a,d,9,f)(1,3)$ &3\\
id  & $(0,1,2,3)$  & $(4,8,b,f)$ &4\\
id  & $(4,f)(d,7)(a,9,5)$  & $(7,c,d)(f,b,4)$ &5\\
id  & $(0,1,2)(a,e)$  & $(f,b)(1,0,2)$ &6\\
id  & $(6,7,c)$  & $(6,c,7)$ &8\\
id  & id &id  &20\\
\hline
\end{tabular}
\end{center}
Hence we have $\{0,1,2,3,4,5,6,8,20\}\subset J_3[16]$.\\

(3) $b_{16}-8 \not \in  J_{3}[16]$:\\
If $b_{16}-8=20-8=12\in J_{3}[16]$, then a 3-way $(v,4,2)$ trade of volume 8 is contained in the $S(2,4,16)$ design. Let $T$ be this trade.\\
If all elements in �found (T) appear 3  times in $T_i$ , then for one block as $a_{1}a_{2}a_{3}a_{4}$, there exist 8 more blocks, so $|T_i|\geq9$.
Hence there exists $x\in$ found (T), with $r_x=2$. Without loss of generality, let $xa_{1}a_{2}a_{3}$ and $xb_{1}b_{2}b_{3} $
be in $T_{1}$. But $T$ is Steiner trade so there exist (for example):  $xb_{1}a_{2}a_{3}$  and  $xa_{1}b_{2}b_{3}$ in $T_2$ and there exist
$xa_{1}b_{3}a_{3}$ and $xb_{1}b_{2}a_{2}$ in $T_3$.
Now $T_1$ must contains at least 6 pairs: $a_{1}b_{2},\ a_{1}b_{3},\ a_{2}b_{1}, \ a_{3}b_{1},\ a_{3}b_{3},\ a_{2}b_{2}$ which those come in disjoint blocks,
since $T$ is Steiner. So we have:
\begin{center}
�\begin{tabular}{|c|c|c|}
\hline
$T_{1}$& $T_{2}$ & $T_{3}$\\
\hline\hline
$xa_{1}a_{2}a_{3}$ & $xb_{1}a_{2}a_{3}$  & $xa_{1}b_{3}a_{3}$\\
$xb_{1}b_{2}b_{3}$& $ xa_{1}b_{2}b_{3}$ & $xb_{1}b_{2}a_{2}$\\
$a_{1}b_2$ & $--$ & $--$ \\
$a_{1}b_3$ & $--$ & $--$ \\
$a_{2}b_1$ & $--$ & $--$ \\
$a_{2}b_2$ & $--$ & $--$ \\
$a_{3}b_1$ & $--$ & $--$ \\
$a_{3}b_3$ & $--$ & $--$ \\
\hline
\end{tabular}
\end{center}
We know that the $S(2,4,16)$ design is unique (See~\cite{1}). Without loss of generality, we can assume  $x,a_{1},a_{2},a_{3}=0,1,2,3$ and $x,b_{1},b_{2},b_{3}=0,4,5,6$ (two blocks of the $S(2,4,16)$ design). Hence $T$ has the following form:
\begin{center}
�\begin{tabular}{|c|c|c|}
\hline
$T_{1}$& $T_{2}$ & $T_{3}$\\
\hline\hline
$0123$ & $0423$  & $0163$\\
$0456$& $ 0156$ & $0452$\\
$15bd$ & $--$ & $--$ \\
$168e$ & $--$ & $--$ \\
$24ce$ & $--$ & $--$ \\
$257f$ & $--$ & $--$ \\
$349d$ & $--$ & $--$ \\
$36af$ & $--$ & $--$ \\
\hline
\end{tabular}
\end{center}
Therefore $r_{7}=1$,  and  by Lemma 3 in~\cite{10} this is impossible .
\end{proof}
\begin{lem}~\label{4.3}
$\{0,23,29,50\}\subseteq J_{3}[25]$.
\end{lem}
\begin{proof}
Construct an  $S(2,4,25)$ design, $(\mathcal{V},\mathcal{B})$ with\\
$\mathcal{V}=\mathcal{Z}_{10}\cup\{a,b,c,d,e,f,g,h,i,j,k,l,m,n,o\}$.
All 50 blocks of $\mathcal{B}$ are listed in the following,
which can be found in Example~1.34 in~\cite{1}.\\
$\{0,1,2,i\},\  \{0,l,3,6\},\ \{0,4,8,o\},\ \{0,a,5,9\},\ \{0,7,g,h\},\ \{0,b,d,n\},\\
 \{0,c,f,g\},\ \{0,k,m,e\},\ \{1,3,a,b\},\ \{1,4,7,m\},\ \{1,5,6,o\},\ \{1,8,f,h\},\\
 \{1,9,e,l\},\ \{1,c,k,n\}, \{1,d,g,j\},\ \{2,3,7,o\},\ \{2,4,b,9\},\ \{2,8,5,n\},\\
 \{2,6,f,g\},\ \{2,a,c,m\},\ \{2,d,k,l\},\ \{2,e,h,j\},\ \{3,4,5,j\},\ \{3,8,c,d\},\\
 \{3,9,k,f\},\ \{3,e,g,n\},\  \{3,h,i,m\},\  \{4,6,d,e\},\ \{4,a,g,k\},\ \{h,4,c,l\},\\
 \{4,i,n,f\},\ \{5,7,c,e\},\ \{5,b,h,k\},\ \{5,d,f,m\}, \ \{5,g,i,l\},\ \{6,7,8,k\},\\
 \{6,9,c,i\},\ \{6,a,h,n\},\ \{6,b,j,m\},\ \{9,7,j,n\},\ \{7,a,d,i\},\  \{7,b,f,l\},\\
 \{m,g,8,9\},\ \{8,a,j,l\},\ \{8,b,e,i\},\ \{9,d,h,o\},\  \{a,e,f,o\},\  \{b,c,g,o\},\\
 \{i,j,k,o\},\ \{l,m,n,o\}$.\\
Consider the following permutations on $\mathcal{V}$.
\begin{center}
�\begin{tabular}{|c|c|c|c|}
\hline
$\pi_{1}$& $\pi_{2}$ & $\pi_{3}$ &  int. no.\\
\hline\hline
id & $(0,1,2,3)$  & $(2,1,0,3)$ &23\\
id & $(0,1,2)$  & $(2,1,0)$ &29\\
id & id   & id &50\\
\hline
\end{tabular}
\end{center}
Hence we have $\{23,29,50\}\subset J_3[25]$.\\
By taking the 5th, 6th and 8th designs, of Table~1.34 in~\cite{1}, we have $0\in J_{3}[25]$.
\end{proof}
\begin{lem}~\label{4.4}
$\{1,63\}\subseteq J_{3}[28]$.
\end{lem}
\begin{proof}
\  $63\in J_{3}[28]$ by taking an $S(2,4,28)$ design thrice.
We obtain $1\in J_3[28]$, by applying the following permutations on the design of Lemma~\ref{6.3} (Step 1) in the last section.\\
 $\pi_1$ is identity, $\pi_2=(0,1,3,5,6,7,12,17,15,18,19,20,11)(25,26,27)$, and $\pi_3=\pi_{2}^{-1}$.
\end{proof}
\begin{lem}~\label{4.5}
There exist three 4-GDDs of type $4^4$ with $i$ common blocks, $i\in \{0,1,2,4,16\}$.
\end{lem}
\begin{proof}
Take the $S(2,4,16)$ design, $(\mathcal{V},\mathcal{B})$ constructed in Lemma~\ref{4.2}. Consider the parallel class $\mathcal{P}=\{\{0,1,2,3\},\{4,8,b,f\},\{5,9,a,e\},\{6,7,c,d\}\}$ as  the groups of GDD to obtain a 4-GDD of type $4^4\ (\mathcal{X},\mathcal{G}, \mathcal{B'})$, where $\mathcal{X}=\mathcal{V},\ \mathcal{G}=\mathcal{P}$ and $\mathcal{B'}=\mathcal{B}\setminus\mathcal{P}$.\\
Consider the following permutations on $\mathcal{X}$, which keep $\mathcal{G}$ invariant.\\
\begin{center}
\begin{tabular}{|c|c|c|c|}
\hline
$\pi_{1}$& $\pi_{2}$ & $\pi_{3}$ & int. no.\\
\hline\hline
id & id  &  id & 16\\
id& $(6,7,c)$  & $(6,c,7)$ &4\\
id &  $(0,1,2)(a,e)$  & $(f,b)(1,0,2)$ &2\\
id & $(4,f)(d,7)(a,9,5)$  & $(7,c,d)(f,b,4)$ &1\\
id & $(0,1,2,3)$  & $(4,8,b,f)$  &0\\
\hline
\end{tabular}
\end{center}
In fact $J_{p3}[16]$ is precisely the intersection sizes of three 4-GDDs of group type $4^4$ having all groups in common.
\end{proof}
\begin{cor}~\label{4.6}
$\{0,1,2,4,16\}\subseteq J_{p3}[16]$.
 \end{cor}
\begin{lem}~\label{4.7}
There exist three 4-GDDs of type $3^5$ with $i$ common blocks, $i\in\{0,1,3,15\}$.
\end{lem}
\begin{proof}
Take the $S(2,4,16)$ design, $(\mathcal{V},\mathcal{B})$ constructed in Lemma~\ref{4.2}. Delete the element 0 from this design to obtain a 4-GDD of type
 $3^5\ (\mathcal{X},\mathcal{G}, \mathcal{B'})$, where $\mathcal{X}= \mathcal{V}\setminus\{0\}$,\\
 $\mathcal{G}=\{\{1,2,3\},\{4,5,6\},\{7,8,9\},\{a,b,c\},\{d,e,f\}\}$ and $\mathcal{B'}=\mathcal{B}\setminus\{B\in\mathcal{B}:0\in B\}$.
Consider the following permutations on $\mathcal{X}$, which keep $\mathcal{G}$ invariant.\\
\begin{center}
\begin{tabular}{|c|c|c|c|}
\hline
$\pi_{1}$& $\pi_{2}$ & $\pi_{3}$ &  int. no.\\
\hline\hline
id & id   & id  &15\\
id & $(a,c)(1,3)(6,5)$  & $(7,9)(d,e)(2,3)(a,c)$ &1\\
id & $(2,3)(5,6)(7,8)(a,c)(d,f)$  & $(1,2)(4,6)(8,9)(a,c)(d,e)$ &0\\
\hline
\end{tabular}
\end{center}
If we delete $d$ then we have a 4-GDD of type
 $3^5\ (\mathcal{X},\mathcal{G}, \mathcal{B'})$, where $\mathcal{X}=\mathcal{V}\setminus\{d\},$\\
 $\mathcal{G}=\{\{6,7,c\},\{2,8,a\},\{1,5,b\},\{3,4,9\},\{0,e,f\}\}$ and
 $\mathcal{B'}=\mathcal{B}\setminus\{B\in\mathcal{B}:d\in B\}$.
When the following permutations act on $\mathcal{X}$ then we obtain 3 as intersection number.\\
$\pi_1=$ identity, $\pi_2=(6,7,c)$, $\pi_3=(6,c,7)$.\\

We have $\{0,1,3,15\}\subseteq J_{f3}[16]$ because $J_{f3}[16]$ is precisely the intersection sizes of
three 4-GDDs of group type $3^5$ having all groups in common.
\end{proof}
\begin{lem}~\label{4.8}
There exist three 4-GDDs of type $3^4$ with $i$ common blocks, $i\in\{0,1,9\}$.
\end{lem}
\begin{proof}
Take the $S(2,4,13)$ design, $(\mathcal{V},\mathcal{B})$ constructed in Lemma~\ref{4.1}. Delete the element 0 from the design to obtain a 4-GDD of type  $3^4\ (\mathcal{X},\mathcal{G}, \mathcal{B'})$, where $\mathcal{X}=\mathcal{V}\setminus\{0\},$��\\
$\mathcal{G}=\{\{1,3,9\},\{2,8,c\},\{4,5,7\},\{6,a,b\}\}$ and $\mathcal{B'}=\mathcal{B}\setminus\{B\in\mathcal{B}:0\in B\}$.
Consider the following permutations on $\mathcal{X}$, which keep $\mathcal{G}$ invariant.\\
\begin{center}
\begin{tabular}{|c|c|c|c|}
\hline
$\pi_{1}$& $\pi_{2}$ & $\pi_{3}$ &  int. no.\\
\hline\hline
id & id  & id &9\\
id & $(a,b)(4,5)$  & $(a,b)(8,c)$ &1\\
\hline
\end{tabular}
\end{center}
If we delete 8 then we have a 4-GDD of type
 $3^4\ (\mathcal{X},\mathcal{G}, \mathcal{B'})$, where
 $\mathcal{X}=\mathcal{V}\setminus\{8\}$,\\
 $\mathcal{G}=\{\{0,2,c\},\{1,5,6\},\{3,7,a\},\{b,4,9\}\}$ and $\mathcal{B'}=\mathcal{B}\setminus\{B\in\mathcal{B}:8\in B\}$.
When the following permutations act on $\mathcal{X}$ then we obtain 0 as intersection number.\\
$\pi_1=$ identity, $\pi_2=(3,7)(c,0,2)(1,6)(9,b)$, $\pi_3=(9,4)(3,7)(0,2)(1,6)$.\\
We obtain $\{0,1,9\}\subset J_{f3}[13]$ because $J_{f3}[13]$ is precisely the intersection sizes of three 4-GDDs of group type $3^4$ having all groups in common.
\end{proof}
\begin{cor}~\label{4,9}
$\{0,1,9\}\subseteq J_{f3}[13]$.
\end{cor}
%
\section{Applying the recursions}
In this section, we prove the main theorem for all $v\geq40$. First we treat the (easier) case $v\equiv1\ (\rm{mod}\ 12)$.
\begin{thm}~\label{5.1}
For any positive integer $v=12u+1,\ u\equiv0,1\ (\rm{mod}\ 4)$ and $u\geq4$, $J_{3}[v]=I_{3}[v]$.
\end{thm}
\begin{proof}
Start from a 4-GDD of type $3^u$ from Lemma~\ref{3.6}. Give each element of the GDD weight 4. By Lemma~\ref{4.5} there exist three 4-GDDs of type $4^4$ with $\alpha$ common blocks, $\alpha\in J_{p3}[16]$. Then apply construction~\ref{3.1} to obtain three 4-GDDs of type $12^u$ with $\sum_{i=1}^{b}{\alpha_{i}}$
common blocks, where $b=\frac{3u(u-1)}{4}$ and $\alpha_{i}\in J_{p3}[16]$, for $1\leq i \leq b$. By construction~\ref{3.2}, filling in the groups by three $S(2,4,13)$ designs with $\beta_{j}(1\leq j \leq u)$ common blocks, we have three $S(2,4,12u+1)$ designs with $\sum_{i=1}^{b}{\alpha_{i}} + \sum_{j=1}^{u}{\beta_{j}}$ common blocks, where $\beta_{j}\in J_{3}[13]$ for $1\leq j \leq u$. It is checked that for any integer $n\in I_{3}[v]$, $n$ can be written as the form of $\sum_{i=1}^{b}{\alpha_{i}} + \sum_{j=1}^{u}{\beta_{j}}$, where $\alpha_{i}\in J_{p3}[16] (1\leq i \leq b)$ and $\beta_{j}\in J_{3}[13] (1\leq j \leq u)$.
\end{proof}
\begin{thm}~\label{5.2}
For any positive integer $v=12u+1,\ u\equiv2,3\ (\rm{mod}\ 4)$ and $u\geq7$, $J_{3}[v]=I_{3}[v]$.
\end{thm}
\begin{proof}
There exists a $(3u+1,\{4,7^{*}\},1)$-PBD from Lemma~\ref{3.7}, which contains exactly one block of size 7. Take an element from the block of size 7. Delete this element to obtain a 4-GDD of type $3^{u-2}6^1$. Give each element of the GDD weight 4. By Lemma~\ref{4.5}, there exist three 4-GDDs of type $4^4$  with $\alpha$ common blocks, $\alpha\in J_{p3}[16]$. Then apply construction~\ref{3.1} to obtain three 4-GDDs of type $12^{u-2}24$ with $\sum_{i=1}^{b}{\alpha_{i}} $ common blocks, where  $b=\frac{3(u^2-u-2)}{4}$ and $\alpha_{i}\in J_{p3}[16]$ for $1\leq i \leq b$. By construction~\ref{3.2}, filling in the groups by three $S(2,4,13)$ designs with $\beta_{j}(1\leq j \leq u-2)$ common blocks, and three $S(2,4,25)$ designs with $\beta$ common blocks,
 we have three $S(2,4,12u+1)$ designs with $\sum_{i=1}^{b}{\alpha_{i}} + \sum_{j=1}^{u-2}{\beta_{j}} +\beta$ common blocks, where $\beta_{j}\in J_{3}[13]$ for $1\leq j \leq u-2$ and $\beta \in J_{3}[25]$. It is checked that for any integer $n\in I_{3}[v]$, $n$ can be written as the form of $\sum_{i=1}^{b}{\alpha_{i}} + \sum_{j=1}^{u-2}{\beta_{j}} +\beta$, where $\alpha_{i}\in J_{p3}[16] (1\leq i \leq b)$,  $\beta_{j}\in J_{3}[13] (1\leq j \leq u-2)$ and $\beta \in J_{3}[25]$.
\end{proof}
\begin{thm}~\label{5.3}
$J_{3}[73]=I_{3}[73]$.
\end{thm}
\begin{proof}
Start from an $S(2,5,25)$ design. Delete an element from this design to obtain a 5-GDD of type $4^6$. Give each element of the GDD weight 3. By Lemma~\ref{4.7},
there exist three 4-GDDs of type $3^5$ with $\alpha$ common blocks, $\alpha\in\{0,1,3,15\}$. Then apply construction~\ref{3.1}, to obtain three 4-GDDs of type $12^6$ with $\sum_{i=1}^{24}{\alpha_{i}}$ common blocks, where $\alpha_{i}\in
\{0,1,3,15\}$ for $1\leq i \leq24$. By construction~\ref{3.2} filling in the groups by three $S(2,4,13)$ designs with $B_{j}(1\leq j\leq6)$ common blocks, $\beta_{j} \in J_{3}[13]$. we have  three $S(2,4,73)$ designs with $\sum_{i=1}^{24}{\alpha_{i}} + \sum_{j=1}^{6}{\beta_{j}}$ common blocks. It is checked that for any integer $n\in I_{3}[73]$, $n$ can be written as the form of $\sum_{i=1}^{24}{\alpha_{i}} + \sum_{j=1}^{6}{\beta_{j}} $.
\end{proof}
For the case $v=12u+4$ we have the following Theorems:
\begin{thm}~\label{5.4}
$I_{3}[40]\setminus\{b_{40}-15,\ b_{40}-14\}\subseteq J_{3}[40]$.
\end{thm}
\begin{proof}
we use of "${v}\rightarrow {3v+1}$" rule, (See~\cite{9}). Let $(\mathcal{V},\mathcal{B})$ be an $S(2,4,v)$ design, and let $\mathcal{V'}$ be a set such that $|\mathcal{V'}|=2v+1$, $\mathcal{V'}\cap \mathcal{V}=\phi$. Let $(\mathcal{V'},\mathcal{C})$ be a resolvable $STS(2v+1)$ and let $\mathcal{R}=\{R_1,\ldots, R_v\}$ be a resolution of $(\mathcal{V'},\mathcal{C})$, that is, let $(\mathcal{V'},\mathcal{C},\mathcal{R})$  be a Kirkman triple system of order $2v+1$; since $v\equiv\ 1,4 \ (\rm mod\ 12)$, such a system exists. Form the set of quadruples
$D_i=\{\{v_i,x,y,z\}: v_i\in\textrm{V}, \{x,y,z\}\in R_i \}$, and put $\mathcal{D}=\bigcup_{i}D_i$.
Then $(\mathcal{V}\cup\mathcal{V'},\mathcal{B}\cup\mathcal{D} )$ is an $S(2,4,3v+1)$ design.\\
Now let $v=13$ and  $(\mathcal{V'},\mathcal{C})$ be a $KTS(27)$ containing three disjoint Kirkman triple systems of order 9. Let $R_{1},\ldots, R_{4}, R_{5}, \ldots, R_{13}$
are the 13 parallel classes of the $KTS(27)$ so that $R_{1},\ldots, R_{4}$ each induce parallel classes in the three $KTS(9)$'s. We add 13 elements $a_{1},\ldots, a_{4},b_{1},\ldots,b_{9}$ to this $KTS(27)$ and form blocks by adding $a_{i}$ to each triple in $R_i \ (i=1,\dots,4)$ and $b_{i}$ to each triple in $R_{i}\ (i\geq5)$. Finally, place an $S(2,4,13)$ design on the 13 new
elements. Consider each ingredient in turn. on the $S(2,4,13)$ design we can get any intersection size from $J_{3}[13]$. On the $(b_{i}, R_{i})$ blocks, we can permute the $R_{i}$ to obtain intersection numbers $\{0,9,18,27,36,45,54,81\}$. We do not have 63 in this set because there exist three designs for intersection and we must permute at least three parallel classes.
On the $(a_{i}, R_{i})$ blocks, we can permute the parallel classes of each of the $KTS(9)$'s to obtain intersection numbers $\{0,3,6,9,12,15,18,21,24,27,36\}$. It is checked that for any integer $n\in I_{3}[40]$, $n$ can be written as the sum of these numbers except $\ b_{40}-16,\ b_{40}-15$, and $b_{40}-14$.\\
For $b_{40}-16$:\\
Start from a 4-GDD of type $3^4$ from Lemma~\ref{3.6}, Give each element of the GDD weight 3. By lemma~\ref{4.8} there exist three 4-GDDs of type $3^4$ with $\alpha$ common blocks, $\alpha\in\{0,1,9\}$.
Then apply construction~\ref{3.1} to obtain three 4-GDDs of type $9^4$ with $\sum_{i=1}^{9}{\alpha_i}$ common blocks, where $\alpha_i\in\{0,1,9\}$ for $1\leq i\leq9$. By construction~\ref{3.3} filling in the groups by three $S(2,4,13)$ designs with $\beta_{j}\ (1\leq j\leq4)$ common blocks, $\beta_j\in J_{3}[13]$. We have three $S(2,4,v)$ designs with $\sum_{i=1}^{9}{\alpha_i}+\sum_{j=1}^{4}{\beta_j}-3$
common blocks. $b_{40}-16$ can be written as this form.
\end{proof}
\begin{thm}~\label{5.5}
 $J_{3}[76]=I_{3}[76]$.
\end{thm}
\begin{proof}
Using a 5-GDD of type $5^5$ ( the 2-(25,5,1) design itself);  Give each element of the GDD weight 3. By Lemma~\ref{4.7}, there exist three 4-GDDs of type $3^5$
with $\alpha$ common blocks, $\alpha\in\{0,1,3,15\}$. Then apply construction~\ref{3.1},  to obtain three 4-GDDs of type $15^5$ with $\sum_{i=1}^{25}{\alpha_{i}}$ common blocks, where $\alpha_{i}\in\{0,1,3,15\}$ for $1\leq i \leq25$. By construction~\ref{3.2}, filling in the groups by three $S(2,4,16)$ designs with $B_{j}(1\leq j\leq5)$ common blocks, we have  three $S(2,4,76)$ designs with $\sum_{i=1}^{25}{\alpha_{i}} + \sum_{j=1}^{5}{\beta_{j}}$ common blocks, where $\alpha_{i}\in \{0,1,3,15\}$ for $1\leq i \leq 5$ and $\beta_{j} \in J_{3}[16]$. It is checked that for any integer $n\in I_{3}[76]$, $n$ can be written as the form of $\sum_{i=1}^{25}{\alpha_{i}} + \sum_{j=1}^{5}{\beta_{j}} $,  except $b_{76}-8,b_{76}-9,b_{76}-10,b_{76}-11,b_{76}-13,b_{76}-21,b_{76}-22,b_{76}-23$. Now we must handle the remaining values. Rees and Stinson (See~\cite{6}) proved that if  $v\equiv1, 4\ (\rm mod \ 12)$,  $w\equiv1, 4\ (\rm mod \ 12)$ and $v\geq3w+1$, then there exists an $S(2,4,v)$ design contains an $S(2,4,w)$  subdesign. By taking all blocks not in the subdesign identically, and three copies of the subdesign intersecting in all but $s$ blocks we have that $b_{v}-s\in J_{3}[v]$ if  $b_{w}-s\in J_{3}[v]$. Using this result with $w=13$, we obtain intersection numbers $b_{v}-8,b_{v}-9,b_{v}-10,b_{v}-11,b_{v}-13$ for $v\geq40$. Similarly using $w=25$, we obtain $b_{v}-21\in J_{3}[v]$ for $v\geq76$.\\
There exists a 4-GDD of type $12^515^1$ from Lemma~\ref{3.8}. Filling in the groups with five $S(2,4,13)$ designs and one $S(2,4,16)$ design. Hence we have an $S(2,4,76)$ design. This design has five $S(2,4,13)$ subdesigns intersecting in a single element.
By choosing suitable intersection sizes from $J_3[13]$ we can obtain $\{b_{76}-22, b_{76}-23\}\subset J_3[76]$.
\end{proof}
\begin{lem}~\label{5.6}
(i) $\{b_{v}-21,b_{v}-22,b_{v}-23,b_{v}-25\}\subset J_{3}[v]$, for $v=52$ and $64$.\\
(ii) $\{b_{v}-22,b_{v}-23,b_{v}-25\}\subset J_{3}[v]$, for $v=88,100,$ and $112$.
\end{lem}
\begin{proof}
\ $\textbf{i}$, for $v=52$, observe that there exists a GDD on 52 elements with block size 4 and group type $13^4$ (See~\cite{2}). Construct three $S(2,4,52)$ designs, take the blocks of GDD identically. Replace each of the four groups by three $S(2,4,13)$ designs. By choosing suitable intersection sizes from $J_3[13]$, we get $\{b_{52}-21,b_{52}-22,b_{52}-23,b_{52}-25\}\subset J_{3}[52]$.\\
Consider $v=64$. Let $G,B$ be a GDD on 21 elements with block size 4 and 5, and group type $5^14^4$ (See~\cite{2}). Apply Lemma~\ref{3.5} to produce $S(2,4,64)$ design. This design has four $S(2,4,13)$ subdesigns intersecting in a single element, and by choosing  suitable intersection sizes from $J_3[13]$. We have $\{b_{64}-21,b_{64}-22,b_{64}-23,b_{64}-25\}\subset J_{3}[64]$.\\
$\textbf{ii}$, There exists a 4-GDD of type $12^{u-1}15^1$ from Lemma~\ref{3.8}, for $u=7,8,$ and 9. Filling in the groups with $S(2,4,13)$ designs and one $S(2,4,16)$ design. Hence we have an $S(2,4,12u+4)$ design. This design has $u-1$, $S(2,4,13)$ subdesigns intersecting in a single element.
By choosing suitable intersection sizes from $J_3[13]$ we can obtain $\{b_{12u+4}-22,b_{12u+4}-23, b_{12u+4}-25\}\subset J_3[12u+4]$, for $u=7,8,$ and 9.
\end{proof}
\begin{thm}~\label{5.7}
For any positive integer $v=12u+4,\ u\equiv0,1\ (\rm{mod}\ 4), u\geq4$, $J_{3}[v]=I_{3}[v]$.
�����\begin{proof}
There exists a 4-GDD of type $3^u$ with $b=\frac{3u(u-1)}{4}$ blocks. By Lemma~\ref{3.4}, we have three $S(2,4,12u+4)$ designs with  $\sum_{i=1}^{b}{\alpha_{i}} + \sum_{j=1}^{u}{\beta_{j}}$ common blocks, where $\beta_{j}+1\in J_{3}[16]$ for $1\leq j \leq u-1$, $\beta_{u}\in J_{3}[16]$ and  $\alpha_{i}\in J_{p3}[16]$ for $1\leq i \leq b$. This produce all values except
$b_{v}-8,b_{v}-9,b_{v}-10,b_{v}-11,b_{v}-13,b_{v}-21,b_{v}-22,b_{v}-23,$ and $b_{v}-25$. By a similar argument as in Lemma~\ref{5.5} we have
$\{b_{v}-8,b_{v}-9,b_{v}-10,b_{v}-11,b_{v}-13\}\subset J_{3}[v](v\geq40)$, and $b_{v}-21\in J_{3}[v]\  (v\geq76)$.\\
 $b_{v}-21\in J_{3}[v]$ for $v=52$ and 64 by Lemma~\ref{5.6}. \\
But $b_{v}-22,b_{v}-23$, and $b_{v}-25$: We know, $I_{3}[40]-\{b_{40}-14, b_{40}-15\}\subseteq J_{3}[40]$, By Ress and Stinson theorem $\{b_{v}-22,b_{v}-23,b_{v}-25\}\subset J_{3}[12u+4]$, for all $u$ that $12u+4\geq 3\times 40+1\Longrightarrow u\geq10$. Hence
it remains to prove that $\{b_{v}-22,b_{v}-23,b_{v}-25\}\subset J_{3}[12u+4]$ for $4\leq u\leq9\ (u\equiv0,1\ (\rm{mod}\ 4))$, that it is proved in Lemma~\ref{5.6}.
\end{proof}
\end{thm}
\begin{thm}~\label{5.8}
For any positive integer $v=12u+4,\ u\equiv2,3\ (\rm{mod}\ 4), u\geq7$, $J_{3}[v]=I_{3}[v]$.
�\begin{proof}
By proof of Theorem~\ref{5.2}, there exists a  4-GDD of type $3^{u-2}6^1$. From Lemma~\ref{3.4},  we have three $S(2,4,12u+4)$ designs with  $\sum_{i=1}^{b}{\alpha_{i}} + \sum_{j=1}^{u-2}{\beta_{j}} + d_{1}$ common blocks, where $\beta_{j}+1\in J_{3}[16]$ for $1\leq j \leq u-3$, $\beta_{u-2}\in J_{3}[16]$,  $\alpha_{i}\in J_{p3}[16]$ for $1\leq i \leq b$ and $d_{1}+1\in J_{3}[28]$. Like the previous case we have all intersection numbers except $b_{v}-22, b_{v}-23$, and $b_{v}-25$:
By a similar argument as in Theorem~\ref{5.7}, we have $\{b_{v}-22,b_{v}-23,b_{v}-25\}\subset J_{3}[12u+4]$, for all $u$ that $12u+4\geq 3\times 40+1\Longrightarrow u\geq10$. Hence
it remains to prove that $\{b_{v}-22,b_{v}-23,b_{v}-25\}\subset J_{3}[12u+4]$ for $7\leq u\leq9\ (u\equiv2,3\ (\rm{mod}\ 4))$, that it is proved in Lemma~\ref{5.6}.
\end{proof}
\end{thm}
%
\section{Small Orders}
Three small orders, $\{25,28,37\}$, remain. We use some techniques to determine situation of half of numbers  that those can be as intersection numbers.
In the next example we discuss a method which may help in understanding a general method in the following theorems.
\begin{exm}~\label{6.1}
\rm{
Construct an  $S(2,4,25)$ design,  $(\mathcal{V},\mathcal{B})$ with
$\mathcal{V}=\mathcal{Z}_{25}$. In this $S(2,4,25)$ design, the elements $\{1,2,3,5,6,8,9\}$ induce an $STS(7)$. All blocks of $\mathcal{B}$ are listed in the following,
which can be found in~\cite{11} (design 17).\\
 We divide these blocks to three parts $A, B$ and $C$. $A$ contains the blocks that induce the $STS(7)$. $B$ contains
the blocks that do not contain any element of the $STS(7)$ and $C$ contains the remanded blocks.\\

$A:$\\
$\begin{array}{lllll}
1,2,3,4&1,5,6,7&1,8,9,10&2,5,8,11\\
2,6,9,14&3,5,9,24&3,6,8,22.& \ \\
\end{array}$\\

$B:$\\
$\begin{array}{lllll}
4,10,16,25&4,11,17,21&4,15,22,24&7,10,19,24\\
7,12,14,18&10,13,14,22&11,14,20,24&7,11,22,23.\\
\end{array}$\\

$C:$\\
$\begin{array}{lllll}
1,11,12,13& 1,14,15,16 & 1,17,18,19 & 1,20,21,22& 1,23,24,25\\
2,7,15,17&2,10,12,20 &2,13,16,23 &2,18,21,24& 2,19,22,25\\
3,7,13,25&3,10,11,18 &3,12,15,19&3,14,21,23&3,16,17,20\\
5,14,4,19&5,10,17,23&5,12,21,25&5,13,15,20&5,16,18,22\\
9,4,13,18&9,7,16,21&9,11,15,25&9,12,17,22&9,19,20,23\\
8,4,7,20&8,12,16,24&8,13,19,21&8,14,17,25&8,15,18,23\\
6,4,12,23&6,10,15,21& 6,11,16,19&6,13,17,24&6,18,20,25.\\
\end{array}$\\

Consider the permutation $\pi=(1,2,3)(18,17,16,13,12,11)$. This permutation consists of two parts the first part $\pi_{1}=(1,2,3)$ contains some elements of the $STS(7)$ and the second part $\pi_{2}=(18,17,16,13,12,11)$ does not contain any element of the $STS(7)$. When $\pi$ and $\pi^{-1}$ act on $A$ we have 1 as intersection number on $A,\ \pi(A),$ and $\pi^{-1}(A)$.\\

\begin{tabular}{c|c|c}
$A$& $\pi(A)$ &$\pi^{-1}(A)$\\
\hline
$\textbf{1,2,3,4}$& $\textbf{1,2,3,4}$ & $\textbf{1,2,3,4}$\\
$1,5,6,7$ & $2,5,6,7$ & $3,5,6,7$\\
$1,8,9,10$ & $2,8,9,10$ & $3,8,9,10$\\
$2,5,8,11$ & $3,5,8,\b{18}$ & $1,5,8,\b{12}$\\
$2,6,9,14$ & $3,6,9,14$ & $1,6,9,14$\\
$3,5,9,24$ & $1,5,9,24$ & $2,5,9,24$\\
$3,6,8,22$ & $1,6,8,22$ & $2,6,8,22$\\
\end{tabular}\\
\\

But when $\pi$ and $\pi^{-1}$ act on $\mathcal{B}\setminus A$, we have 6 as intersection number on $\mathcal{B}\setminus A,\ \pi(\mathcal{B}\setminus A),$ and $\pi^{-1}(\mathcal{B}\setminus A)$. Then we get intersection number $7=1+6$ on $\mathcal{B},\ \pi(\mathcal{B})$ and $\pi^{-1}(\mathcal{B})$.\\
\\

$\mathcal{B}\setminus A:$\\
~$\begin{array}{lllll}
1,11,12,13 & 1,14,15,16& 1,17,18,19& 1,20,21,22& 1,23,24,25\\
2,7,15,17 &2,10,12,20& 2,13,16,23& 2,18,21,24& 2,19,22,25\\
3,7,13,25 & 3,10,11,18 & 3,12,15,19& 3,14,21,23& 3,16,17,20\\
\textbf{5,14,4,19} &5,10,17,23& 5,12,21,25& 5,13,15,20& 5,16,18,22\\
9,4,13,18 &9,7,16,21& 9,11,15,25& 9,12,17,22& \textbf{9,19,20,23}\\
\textbf{8,4,7,20} &8,12,16,24& 8,13,19,21& 8,14,17,25& 8,15,18,23\\
6,4,12,23& \textbf{6,10,15,21}& 6,11,16,19& 6,13,17,24& 6,18,20,25.\\
4,10,16,25 &4,11,17,21& \textbf{4,15,22,24}& \textbf{7,10,19,24}& 7,11,22,23.\\
7,12,14,18 &10,13,14,22& 11,14,20,24 & \  & \ .\\
\end{array}$\\

$\pi(\mathcal{B}\setminus A):$\\
$\begin{array}{lllll}
\b{2},18,11,12&\b{2},14,15,13&\b{2},16,17,19&\b{2},20,21,22&\b{2},23,24,25\\
\b{3},7,15,16&\b{3},10,11,20&\b{3},12,13,23&\b{3},17,21,24&\b{3},19,22,25\\
\b{1},7,12,25&\b{1},10,18,17&\b{1},11,15,19&\b{1},14,21,23&\b{1},13,16,20\\
6,4,11,23&\textbf{6,10,15,21}&6,18,13,19&6,12,16,24&6,17,20,25\\
\textbf{5,14,4,19}&5,10,16,23&5,11,21,25&5,12,15,20&5,13,17,22\\
8,11,13,24&8,12,19,21&8,14,16,25&8,15,17,23&\textbf{8,4,7,20}\\
9,4,12,17&9,7,13,21&9,18,15,25&9,11,16,22&\textbf{9,19,20,23}\\
4,10,13,25&4,18,16,21&\textbf{4,15,22,24}&\textbf{7,10,19,24}&7,18,22,23\\
7,11,14,17&10,12,14,22& 18,14,20,24 & \  & \ .\\
\end{array}$\\

$\pi^{-1}(\mathcal{B}\setminus A):$\\
$\begin{array}{lllll}
\b{3},12,13,16&\b{3},14,15,17 &\b{3},18,11,19&\b{3},20,21,22&\b{3},23,24,25\\
\b{1},7,15,18&\b{1},10,13,20&\b{1},16,17,23&\b{1},11,21,24&\b{1},19,22,25\\
\b{2},7,16,25&\b{2},10,12,11&\b{2},13,15,19&\b{2},14,21,23&\b{2},17,18,20\\
6,4,13,23&\textbf{6,10,15,21}&6,12,17,19&6,16,18,24&6,11,20,25\\
\textbf{5,14,4,19}&5,10,18,23&5,13,21,25&5,16,15,20&5,17,18,22\\
8,13,17,24&8,16,19,21&8,14,18,25&8,15,11,23&\textbf{8,4,7,20}\\
9,4,16,11&9,7,17,21&9,12,15,25&9,13,18,22&\textbf{9,19,20,23}\\
4,10,17,25&4,12,18,21&\textbf{4,15,22,24}&\textbf{7,10,19,24}&7,12,22,23\\
7,13,14,11&10,16,14,22&12,14,20,24 & \ & \ .\\
\end{array}$\\

In fact we obtain two intersection numbers, the first number is obtained  when $\pi_{1}$ and $\pi_{1}^{-1}$ act on $A$. The second number
is obtained when $\pi_{2}$ and $\pi_{2}^{-1}$ act on $\mathcal{B}\setminus A$. Then we add these numbers and obtain intersection number of $\mathcal{B},\ \pi(\mathcal{B}),$ and $\pi^{-1}(\mathcal{B})$. Since the common blocks of $A,\ \pi_{1}(A),$ and $\pi_{1}^{-1}(A)$ do not contain any element of $\pi_{2}$ and common blocks of $\mathcal{B}\setminus A,\ \pi_{2}(\mathcal{B}\setminus A),$ and $\pi_{2}^{-1}(\mathcal{B}\setminus A)$ do not contain any element of $\pi$.\\
Note, we choose some permutations which change at most two elements of each block. Also the design is Steiner, so when the block $b$ changes, it is commuted to different block from the other blocks. Hence by applying permutations, no new common block form in $\pi(\mathcal{B})$ and $\pi^{-1}(\mathcal{B})$. (We separate $B$ and $C$ for choosing suitable permutations.)}
\end{exm}
\begin{lem}
$[0,11] \cup \{13,15,17,20,29,50\} \cup [22,24] \subseteq J_{3}[25]$ and $\{42\}\not \in J_{3}[25]$.
\end{lem}
\begin{proof}
Take the design $(\mathcal{V},\mathcal{B})$ which is stated in Example~\ref{6.1}. We get these intersection numbers $[5,8]\cup\{11,13,15,17,20,24,29\}$ by the method of Example~\ref{6.1} on $(\mathcal{V},\mathcal{B})$.
Also we get these intersection numbers: $[0,10]\cup \{22,23,29\}$, with applying straight permutations on $\mathcal{B}$. \\
We have $42\notin J_3[25]$ since $42\notin J_2[25]$. This completes proof.
\end{proof}
\begin{lem}\label{6.3}
 $\ [1,24]\cup \{27,28,33,37,39,63\}\subseteq J_3[28]$.
\end{lem}
\begin{proof}
We obtain these intersection numbers in two steps.\\
\textbf{Step 1:}\\
Construct an  $S(2,4,28)$ design, $(\mathcal{V},\mathcal{B})$ with
$\mathcal{V}=\mathcal{Z}_{28}$. All blocks of $\mathcal{B}$ are listed in the following,
which can be found in Theorem 20 in \cite{8}. In this $S(2,4,28)$ design the elements $\{2,4,16,22,25,26,27\}$ induce an $STS(7)$.\\
By a similar argument in Example~\ref{6.1}, we obtain these intersection numbers: $[2,7]\cup[10,12]\cup [16,19]\cup[21,24]\cup \{14,27,28,33,37,39\}$.\\

$\begin{array}{lllll}
2,0,1,3& 4,5,6,7 & 16,17,18,19 & 22,20,21,23 & 25,7,8,1\\
2,7,18,23 & 4,11,14,23  & 16,3,11,13 & 22,3,6,19 & 25,5,14,18\\
2,10,12,17& 4,0,8,12 & 16,0,5,20 & 22,5,10,15 & 25,9,19,20\\
2,13,20,24 & 4,1,17,21 & 16,1,6,23 & 22,0,7,17 & 25,0,15,23\\
2,5,19,21 & 4,15,19,24 & 16,9,15,21 & 22,9,12,18 & 25,6,13,17\\
2,8,6,15 & 4,3,18,20 & 16,7,12,24 & 22,1,14,24 & 25,3,12,21\\
26,1,15,18& 27,3,15,17 & 8,9,10,11 & 0,6,9,24 & 24,25,26,27\\
26,7,10,19& 27,5,12,23 & 12,13,14,15 & 0,10,13,18 & 9,2,4,27\\
26,9,17,23& 27,1,10,20 & 3,7,9,14 & 1,11,12,19 & 11,2,22,25\\
26,3,8,5& 27,0,14,19 & 3,10,23,24 & 6,10,14,21 & 14,2,16,26\\
26,12,6,20& 27,11,6,18 & 5,11,17,24 & 7,11,15,20 & 10,4,16,25\\
26,0,11,21& 27,7,13,21 & 1,5,9,13 & 8,14,17,20 & 13,4,22,26\\
8,16,22,27& 8,13,19,23& 8,18,21,24. & \  & \ \\
\end{array}$\\

\textbf{Step 2:}\\
Take an  $S(2,4,28)$ design, $(\mathcal{V},\mathcal{B})$ with
$\mathcal{V}=\mathcal{Z}_{28}$. In this $S(2,4,28)$ design, the elements $\{4,5,6,13,14,15,19,20,21\}$ induce an $STS(9)$. All blocks of $\mathcal{B}$ are listed in the following,
which can be found in Theorem 21 in \cite{8}.\\
$\begin{array}{lllll}
4,0,1,7&5,0,2,8 &6,10,12,25 &13,0,10,16 &14,12,16,22\\
4,2,3,12&5,1,3,10&6,8,22,23&13,1,8,9&14,8,10,24\\
4,8,17,25&5,9,18,26&6,7,16,27&13,2,17,27&14,3,18,25\\
4,9,23,24&5,7,22,24&6,1,2,11&13,7,12,23&14,2,7,9\\
4,10,11,26&5,11,12,27&6,0,3,9&13,11,18,24&14,0,11,17\\
15,0,12,18&19,0,22,25&20,8,18,27&21,0,27,24&1,17,18,22\\
15,3,7,8&19,1,12,24&20,7,11,25&21,3,11,23&1,23,25,27\\
15,1,16,26&19,8,11,16&20,9,12,17&21,7,10,18&2,24,25,26\\
15,9,11,22&19,9,10,27&20,2,10,22&21,8,12,26&3,22,26,27\\
15,10,17,23&19,7,17,26&20,0,23,26&21,9,16,25&3,17,16,24\\
1,14,20,21&3,13,19,20&4,5,16,20&5,13,15,25&4,13,21,22\\
2,15,19,21&4,6,18,19&4,14,15,27&6,13,14,26&5,14,19,23\\
2,18,23,16&5,6,17,21&6,15,20,24.& \ &\ \\
\end{array}$

By previous method we obtain these intersection numbers $[5,17]\cup[19,21]\cup \{23,24,28,33,39\}$, in this step.\\
Also we obtain 1 as intersection number in Lemma~\ref{4.4}.
\end{proof}
\begin{lem}
$\{18,19,78,79,81,87,102,103,111\} \cup [21,32] \cup [34,36] \cup [38,43] \cup [45,48] \cup [52,54] \cup [58,63] \cup [67,71] \subseteq J_{3}[37]$
\end{lem}
\begin{proof}
In this Lemma we have three steps.\\
\textbf{Step 1:}\\
Take an  $S(2,4,37)$ design, $(\mathcal{V},\mathcal{B})$ with\\
$\mathcal{V}=\{a_0,\cdots, a_8,b_0,\cdots, b_8, c_0,\cdots, c_8,d_0,\cdots, d_8,\infty\}$. Develop the  following base blocks over $\mathcal{Z}_{9}$
to obtain all blocks of $\mathcal{B}$ (See~\cite{7}). In this $S(2,4,37)$ design the elements
 $\{a_{0},a_{3},a_{6}, b_{0},b_{3},b_{6},c_{0},c_{3},c_{6}\}$ induce
an $STS(9)$.\\

\noindent$\{\infty,a_{0},a_{3},a_{6}\},\ \{\infty,b_{0},b_{3},b_{6}\},\ \{\infty,c_{0},c_{3},c_{6}\},\ \{\infty,d_{0},d_{3},d_{6}\}$\\
$\{a_0,a_1,b_3,c_0\},\ \{a_0,a_5,b_6,c_6\},\ \{a_0,a_7,d_0,d_1\},\ \{a_0,b_0,b_4,c_3\}$\\
$\{a_1,c_3,c_8,d_0\},\ \{a_2,c_6,c_7,d_0\},\ \{a_3,b_8,d_0,d_7\},\ \{a_4,b_2,b_3,d_0\}$\\
$\{b_0,c_1,d_0,d_5\},\ \{b_5,b_7,c_0,d_0\},\ \{b_6,c_2,c_4,d_0\}$.\\

 By a similar argument in Example~\ref{6.1}, we obtain these intersection numbers:\\
  $\{18,19,21,22,69,70,78,81\}\cup [24,32]\cup[34,36]\cup[38,43]\cup[45,48]\cup[52,54]\cup[60,62]$.\\
 \textbf{Step 2:}\\
 Construct an  $S(2,4,37)$ design, $(\mathcal{V},\mathcal{B})$ with\\
$\mathcal{V}= \mathcal{Z}_{11}\times\{1,2,3\}\cup\{\infty_{1}, \infty_{2},\infty_{3},\infty_{4}\}$. In this $S(2,4,37)$ design the elements $\{0_1,1_1,2_2,10_{2}, 3_{3},4_{3},5_{3}\}$ induce an $STS(7)$. Develop the  following base blocks over $\mathcal{Z}_{11}$
to obtain all blocks of $\mathcal{B}$ ( $\infty_{1}, \infty_{2},\infty_{3}$ and $\infty_{4}$ are constants) (See~\cite{8}). By a similar argument in Example~\ref{6.1}, we obtain these intersection numbers: \\
$\{23,26,29,32,35,36,38,39,42,43,45,47,48,53,54,60,61,68,69,78\}$.\\

\noindent$\{0_1,0_2,0_3,\infty_{1}\}$, $\{0_1,1_2,2_3,\infty_{2}\}$, $\{0_1,2_2,5_3,\infty_{3}\}$ \\
$\{0_1,8_2,6_3,\infty_{4}\}$, $\{0_1,1_1,5_1,10_{2}\}$, $\{0_2,2_2,5_2,7_{3}\}$ \\
$\{8_1,0_3,1_3,5_{3}\}$, $\{0_1,3_1,6_2,7_{2}\}$, $\{0_2,4_2,8_3,10_{3}\}$. \\
$\{2_1,4_1,0_3,3_{3}\}$ \\
Also the design contains the block $\{\infty_{1}, \infty_{2},\infty_{3},\infty_{4}\}$.\\

\textbf{Step 3:}
Take an $S(2,4,37)$ design,  $(\mathcal{V},\mathcal{B})$ with
$\mathcal{V}=\{\infty\}\cup(\{x,y,z\}\times \mathcal{Z}_{12})$. By developing the following base blocks over $\mathcal{Z}_{12}$ we get the main part of the blocks (See \cite{2}):\\

\noindent$\{z_0,x_0,y_0,\infty\}$, $\{x_0,x_4,y_{11},z_5\}$, $\{x_2,z_0,z_1,z_5\}$ \\
$\{x_7,y_0,y_1,z_9\}$, $\{x_{10},y_0,y_2,z_{4}\}$, $\{x_3,y_0,y_4,z_{7}\}$ \\
$\{x_2,y_0,y_5,z_{10}\}$, $\{x_5,y_1,z_0,z_{2}\}$. \\
and the short orbits:\\
$\{y_0,y_3,y_6,y_{9}\},\ $
$\{z_0,z_3,z_6,z_{9}\}$. \\

Call the resulting set of 102 blocks $B$ and call the other set of blocks $C$. $C$ contains nine blocks which covers the remaining pairs.
In fact $C$ comes from $S(2,4,13)$ design with omitting one flower. This enable us to replace $C$ by a different set $C'$ or $C''$ of blocks covering the same pairs,
So in this part we can have intersection number  $C\cap C'\cap C''$. Recall that $C\cap C'\cap C''$ can be any of $\{0,1,9\}\subseteq J_{f3}[13]$.
Also we consider some permutations on $B$ which be used in~\cite{2} and those are suitable for three designs. Let $\pi$ be one of them. We construct $B'=\pi(B)$ and $B''=\pi^{-1}(B)$.
Hence we obtain intersection sizes $|B\cap {B'}\cap B''|+i,\ i\in\{0,1,9\}$.
Now we get in this step these intersection numbers $\{58,59,62,63,67,78,79,87,102,103,111\}\cup[69,71]$.
\end{proof}

\section{conclusion}
In this paper, we have obtained the complete solution of the intersection problem for three $S(2,4,v)$ designs with $v=13,16$ and $v\geq49$.\\
\textbf{Proof of Theorem~\ref{1.1}:}\\
 (1): By Lemma~\ref{2.1} we have $J_{3}[v]\subseteq I_{3}[v]$.\\
 (2): By combining the results of Theorems~\ref{5.1}, ~\ref{5.2}, ~\ref{5.3}, ~\ref{5.5}, ~\ref{5.7}, and ~\ref{5.8} we have
 $J_{3}[v]=I_{3}[v]$ for all admissible $v\geq49$.\\
   (3): By Theorem~\ref{5.4}, we obtain  $I_{3}[40]\setminus\{b_{40}-15, b_{40}-14\}\subseteq J_{3}[40]$.\\
   (4): It holds by Lemmas~\ref{4.1} and~\ref{4.2}.\\
   (5), (6), and (7): We prove these sentences in the last section.


\begin{thebibliography}{10}
\bibitem{22}
E.~Billington, M.~Gionfriddo, and C.~C. Lindner, \emph{The intersection problem for {$K_{4}-e$} designs}, J. Statist. Plann. Inference \textbf{58} (1997), 5--27.
\bibitem{5}
A.~E. Brouwer, \emph{Optimal packings of {$K_{4}$}'s into a {$K_{n}$}}, J. Combin. Theory Ser. A \textbf{26} (1979), no.~3, 278--297.
\bibitem{14}
Y.~Chang, T.~Feng, and G.~Lo~Faro, \emph{The triangle intersection problem for {$S(2,4,v)$} designs}, Discrete Math. \textbf{310} (2010), no.~22, 3194--3205.
\bibitem{23}
Y.~Chang, T.~Feng, G.~Lo~Faro, and A.~Tripodi, \emph{The fine triangle intersection problem for kite systems}, Discrete Math. \textbf{312} (2012), no.~3, 545-553.
\bibitem{24}
Y.~Chang, T.~Feng, G.~Lo~Faro, and A.~Tripodi, \emph{The triangle intersection numbers of a pair of disjoint {$S(2,4,v)$s}}, Discrete Math. \textbf{310} (2010), no.~21, 3007-3017.
\bibitem{25}
Y.~Chang, and G.~Lo~Faro, \emph{Intersection nunmber of {K}irkman triple systems,} J. Combin. Theory Ser. A \textbf{86} (1999), no.~2, 348-361.
\bibitem{27}
Y.~Chang, and G.~Lo~Faro, \emph{The flower intersection problem for {K}irkman triple systems}, J. Statist. Plann. Inference \textbf{110} (2003), no.~1-2,159--177.
\bibitem{2}
C.~J. Colbourn, D.~G. Hoffman, and C.~C. Lindner, \emph{Intersections of  {$S(2,4,v)$} designs}, Ars Combin. \textbf{33} (1992), 97--111.
\bibitem{1}
G.~Ge, \emph{{G}roup {D}ivisible {D}esigns, in: {H}andbook of combinatorial designs}, second ed., Discrete Mathematics and its Applications, C. J. Colbourn, and J. H. Dinitz (eds.), Chapman \& Hall/CRC, Boca Raton, FL, 2007, pp. 255--260.
\bibitem{26}
M.~Gionfriddo, and C.~C. Lindner, \emph{Construction of  {S}teiner quadruple systems having a prescribed number of blocks in common}, Discrete Math. \textbf{34} (1981), 31--42.
\bibitem{10}
H.~L. Hwang, \emph{On the structure of {$(v,k,t)$} trades}, J. Statist. Plann. Inference \textbf{13} (1986), no.~2, 179--191.
\bibitem{7}
V.~Kr{\v{c}}adinac, \emph{Some new {S}teiner 2-designs {$S(2,4,37)$}}, Ars Combin. \textbf{78} (2006), 127--135.
\bibitem{28}
G.~Lo~Faro, \emph{{S}teiner quadruple systems having a prescribed number of blocks in common}, Discrete Math. \textbf{58} (1986), no.~2, 167--174.
\bibitem{8}
M.~Meszka and A.~Rosa, \emph{Embedding {S}teiner triple systems into {S}teiner systems {$S(2,4,v)$}}, Discrete Math. \textbf{274} (2004), no.~1-3, 199--212.
\bibitem{4}
S.~Milici and G.~Quattrocchi, \emph{On the intersection problem for three {S}teiner triple systems}, Ars Combin. \textbf{24} (1987), no.~A, 175--194.
\bibitem{15}
S.~Rashidi and N.~Soltankhah, \emph{On the possible volume of 3-way trade}, submitted.
\bibitem{6}
R.~Rees and D.~R. Stinson, \emph{On the existence of incomplete designs of block size four having one hole}, Utilitas Math. \textbf{35} (1989), 119--152.
\bibitem{9}
C.~Reid and A.~Rosa, \emph{{S}teiner systems {$S(2,4,v)$} - a survey}, The Electronic Journal of Combinatorics (2010), DS18.
\bibitem{11}
E.~Spence, \emph{The complete classification of {S}teiner systems {$S(2,4,25)$}}, J. Combin. Des. \textbf{4} (1996), no.~4, 295--300.

\end{thebibliography}
\end{document}